\documentclass[preprint,12pt]{elsarticle}

\usepackage{amssymb}
\usepackage{amsmath}
\usepackage{amsthm}
\usepackage{tikz}
\usepackage{latexsym}
\usepackage{url}

\theoremstyle{plain}
\newtheorem{thm}{Theorem}
\newtheorem{lemma}[thm]{Lemma}
\newtheorem{prop}[thm]{Proposition}
\newtheorem{cor}[thm]{Corollary}

 \theoremstyle{definition} 
\newdefinition{question}[thm]{Question}
\newdefinition{conjecture}[thm]{Conjecture}
\newdefinition{defn}[thm]{Definition}

\newdefinition{rmk}[thm]{Remark}

\journal{Journal of Group Theory}

\begin{document}

\begin{frontmatter}
 \title{Generating Sequences of $\mathrm{PSL}(2,p)$}
 
 \author[rvt]{Benjamin Nachman}
 \ead{bpn7@cornell.edu}
 
 \address[rvt]{Cornell University, 310 Malott Hall, Ithaca, NY 14853 USA}
 
 \begin{abstract}
Julius Whiston and Jan Saxl~\cite{W&S} showed that the size of an irredundant generating set of the group $G=\mathrm{PSL}(2,p)$ is at most four and computed the size $m(G)$ of a maximal set for many primes.  We will extend this result to a larger class of primes, with a surprising result that when $p\not\equiv\pm 1\mod 10$, $m(G)=3$ except for the special case $p=7$.   In addition, we will determine which orders of elements in irredundant generating sets of $\mathrm{PSL}(2,p)$ with lengths less than or equal to four are possible in most cases.  We also give some remarks about the behavior of $\mathrm{PSL}(2,p)$ with respect to the replacement property for groups.  
\end{abstract}
 
 
 \end{frontmatter}

\section{Introduction}

The two dimensional projective linear group over a finite field with $p$ elements, $\mathrm{PSL}(2,p)$ has been extensively studied since Galois, who constructed them and showed their simplicity for $p>3$~\cite{Wilson}.   One of their nice properties is due to a theorem by E. Dickson, which shows that there are only a small number of possibilities for the isomorphism types of maximal subgroups.  There are two possibilities: ones that exist for all $p$ and ones that exist for exceptional primes.  A more recent proof of Dickson's Theorem can be found in~\cite{MS} and a complete proof due to Dickson is in~\cite{LD}.  

\begin{thm}[Dickson] The maximal subgroups of $\mathrm{PSL}(2,p)$ are isomorphic to one of the following groups:

\begin{enumerate}
\item $G_p$
\item $\mathrm{D}_{p-1}$, the Dihedral Group of order $p-1$
\item $\mathrm{D}_{p+1}$
\item $\mathrm{A}_4,$ $\mathrm{S_4}$ or $\mathrm{A_5}$,
\end{enumerate}

where $G_p$ is the Frobenius group of order $p(p-1)/2$ that has a natural description as
the semi-direct product $\mathbb{Z}_p \rtimes (\mathbb{Z}_p^*)^2$.  Moreover, while subgroups of types $(1),(2)$ and $(3)$ always exist for $p\neq 2$ (then only (2) and (3) exist), a maximal subgroup isomorphic to $\mathrm{S_4}$ exists if and only if $p\equiv \pm 1\mod 8$, subgroups isomorphic to $\mathrm{A_5}$  exist if and only if $p\equiv\pm 1\mod 10$ and subgroups isomorphic to $\mathrm{A}_4$ are maximal if and only if $p\equiv 3,13,27,37\mod 40$.

\end{thm}

The exceptional maximal subgroups are thus $\mathrm{A}_4$, $\mathrm{S_4}$ and $\mathrm{A_5}$.  Whiston and Saxl~\cite{W&S} have shown that these exceptional groups play a crucial role in describing the size of generating sets.  A generating set $\{g_i\}$ or sequence $\{g_i\}_{i\in I}$ for the group $G$ is called {\it irredundant}\footnote{In other places in the literature, this same property is called {\it independent}.} if after removing any $g_j$ from the set or sequence, the new collection no longer generates $G$.   We will denote\footnote{This function has also been denoted $\mu(G)$.} by $m(G)$ the maximum length of an irredundant generating set of $G$.   In response to Whiston's description of $m(G)$ for $\mathrm{S}_n$~\cite{JW}, Cameron and Cara described the irredundant generating sequences of maximal length~\cite{CC}.  In the same spirit, we will describe which elements can appear in generating sequences of any length up to the maximal length in $\mathrm{PSL}(2,p)$ in most cases.   This size has been determined by Whiston and Saxl~\cite{W&S} for all primes such that the exceptional groups $\mathrm{S_4}$ and $\mathrm{A_5}$ are not maximal in $\mathrm{PSL}(2,p)$.

\begin{thm}[Whiston and Saxl] Let $G=\mathrm{PSL}(2,p)$, $p$ prime.  Then, $m(G)=3$ or $4$.  If $p\neq \pm 1\mod 10$ and $p\neq \pm 1\mod 8$, then $m(G)=3$.  
\end{thm}

In their paper~\cite{W&S}, Whiston and Saxl note that $m(\mathrm{PSL(2,7)})=m(\mathrm{PSL(2,11)})=4$.  Since then, various computations have been made to show that the maximal length is also four when $p=19$ and $p=31$.  The conjecture in~\cite{NachmanThesis} is that this small list of primes constitutes the entire collection.  The strategy presented here does not easily extend to the case of $p\equiv\pm 1\mod 10$, but a large part of this suprising conjecture is proved in this paper, summarized in the following theorem.

\begin{thm}
 \label{thm:main}
 Let $G=\mathrm{PSL}(2,p)$, $p$ prime.  If $p\not\equiv\pm 1\mod 10$,  then, $m(G)=3$ unless $p=7$, in which case $m(G)=4$.
\end{thm}

To begin, we introduce the idea of the replacement property for groups and show how it is useful for constructing irredundant generating sequences.  

\section{Irredundant Generating Sets and the Replacement Property}

Linear algebra often forms a concrete base upon which intuition is built for studying more general objects.  Understanding generating sequences of groups is no exception -- the idea of an irredundant generating set is an analogy to a basis of a vector space.   In the case of vector spaces, the classification of bases is easy -- they all have the same length.  This is not the case for groups.  If we denote by $r(G)$ the minimum length of an irredundant generating sequence, then clearly $m(G)\geq r(G)$ and in general this inequality is strict.  For example, one can easily show based on elementary linear algebra that $G=\mathrm{PSL}(2,p)$ can be generated by two elements so $r(G)=2$ since $G$ is not cyclic.  On the other hand for $p>2$, $|G|$ must be even since its order is $p(p-1)(p+1)/2$.  Therefore, there exist nontrivial elements of order $2$.  Let $H$ be the subgroup generated by all such elements.  Then, $1<H\unlhd G$ and so $H=G$ because $\mathrm{PSL}(2,p)$ is simple.  Furthermore, two elements of order $2$ generate a dihedral group and so there must exist an irredundant generating sequence of length at least $3$ (with elements all of order $2$) and so $r(G)<m(G)$ for $p>2$.  

In addition, for a vector space $V$, every linearly independent subset has length at most $\mathrm{dim}(V)$.   For a group, it is not the case that for $H< G$, $m(H)< m(G)$.  For example, for $G=\mathrm{PSL(2,17)},$ $m(G)=3$ but $G$ has maximal subgroups isomorphic to $\mathrm{S}_4$, for which $m=3$ as well.  A group which does have the property that for all subgroups $H< G$, $m(H)<m(G)$ is called {\it strongly flat}.  Two important examples are $\mathrm{S_4}$~ and ~$\mathrm{A_5}$.  In fact, all the symmetric groups, for which $m(\mathrm{S}_n)=n-1$, are strongly flat~\cite{W&S}.

Another important aspect of vector spaces is the elementary fact that any linearly independent set can replace a segment of a basis.  The idea is to generalize this notion to arbitrary groups.  Instead of looking at bases, the generalization is generating sets.  Also, instead of replacing many elements of the generating set, the focus will be on replacing a single element.  This led D. Collins and R. K. Dennis to make the following definition:

\begin{defn}[Replacement Property] A group $G$ satisfies the {\it replacement property} for the generating sequence $s = (g_1,...,g_k)$ if for any $\mathrm{id}\neq g\in G$, there exists an $i$ so that $s' = (g_1,...,g_{i-1},g,g_{i+1}...,g_k)$ generates $G$.
\end{defn}


A group $G$ is said to satisfy the replacement property if it satisfies the replacement property for all irredundant sequences of length $m(G)$.  Vector spaces satisfy the replacement property, but this is not true for all groups.  For example, consider $G=Q_8$, the Quaternion group.  If we think of $G$ as the elements $\{\pm1,\pm i,\pm j\pm k\}$, then it is clear that $i,j$ is a generating sequence of $G$.  However, we cannot replace either of $i$ or $j$ in this sequence with $-1$.  More generally if the Frattini subgroup of a group is nontrivial, then the nontrivial non-generating elements will cause $G$ to fail the replacement property.  For $Q_8$, $\{\pm 1\}$ is the Frattini subgroup and so it fails the replacement property.  One could modify the definition of the replacement property to exclude such cases.  Either way, there are examples of groups which are Frattini free and still fail the replacement property.  For example, when $p\equiv +1\mod 8$, $\mathrm{PSL}(2,p)$~is such a group.  Before showing this, the definition of replacement property must be reworked slightly.  This property has been phrased in terms of generating sequences, but it can be restated in terms of certain sets of maximal subgroups.  

Let $(M_1,...,M_n)$ be a sequence of maximal subgroups of a finite group $G$ and let $(g_1,...,g_n)$ be a sequence of elements of $G$.  These two sequences are said to correspond to each other if $g_i\not\in M_i$ for any $i\in \{1,...,n\}$ but $g_j\in M_i$ whenever $j\neq i$.   With this connection, there is a relationship between maximal subgroups and irredundant generating sequences:

\begin{prop}[D. Collins and R. K. Dennis] 
If $(g_1,...,g_n)$ is an irredundant generating sequence, then it corresponds to a sequence of maximal subgroups $(M_1,...,M_n)$ and $\cap_{i\in J} M_i\subsetneq\cap_{i\in K} M_i $ for all $J,K\subset I=\{1,...,n\}$ and $K\subsetneq J$.  We say that subgroups with this last property are in general position.
\end{prop}

\begin{proof}
Let $H_i=\langle g_1,...,g_{i-1},g_{i+1},...,g_n\rangle$.  Since $(g_1,...,g_n)$ is an irredundant generating sequence, $H_i$ is a proper subgroup of $G$.  Therefore, there exists a maximal subgroup $H_i\leq M_i$.  Note that $g_i\not\in M_i$, since $M_i$ is also a proper subgroup, but $g_j\in M_i$ for all $j\neq i$ by construction.  Therefore, $(M_1,...,M_n)$ corresponds to $(g_1,...,g_n)$.  Now, one needs to show that the maximal subgroups are in general position.  By construction, for $J\subset I=\{1,...,n\}$ then $g_j\in \cap _{i\in J} M_i$ if and only if $j\not\in J$.  Therefore, the subgroups $\cap_{i\in J} M_i$ are all distinct as no two of them intersect $\{g_1,...,g_n\}$ in the same way.  
\end{proof}

Now that a relationship exists between irredundant generating sequences and maximal subgroups in general position, one can construct a criteria on maximal subgroups for establishing the replacement property.  Using the same ideas as in the previous proposition, one can prove the following:

\begin{prop}[D. Collins and R. K. Dennis] 
Let $s=(g_1,...,g_n)$ be an irredundant generating sequence of the group $G$.  If every sequence of maximal subgroups $(M_1,...,M_n)$ corresponding to $s$ intersects trivially, then $s$ satisfies the replacement property.
\label{important}
\end{prop}

\begin{proof}
We prove the contraposition.  If $s$ fails the replacement property for $g$, then for each $i$, the sequence $(g_1,...,g_{i-1},g,g_{i+1},...,g_n)$ generates a proper subgroup $H_i$ of $G$.  Pick a maximal subgroup $H_i\leq M_i$.  Then, $(M_1,...,M_n)$ corresponds to $s$ by definition and furthermore, $g\in\cap M_i$ by construction.
\end{proof}

Now, we will focus on irredundant generating sequences of $\mathrm{PSL}(2,p)$~which will eventually lead us  to study how this group behaves with respect to the replacement property.

\section{Irredundant Sequences of Maximal Length in~$G=\mathrm{PSL}(2,p)$}

The general strategy for proving that $m(G)=3$ for most cases is to take irredundant generating sequences and try to `glue them together' and see what possibilities exist for the resulting group.   We will make this procedure more quantitative as the discussion progresses.    In this process, we will switch back and forth between considering elements and (maximal) subgroups corresponding to the elements.   Let $g_1,g_2,g_3,g_4\in G$ be an irredundant generating set.  Let $H_1,H_2,H_3,H_4$ be the corresponding family of subgroups in general position, (e.g. $H_1=\langle g_2,g_3,g_4\rangle$) and let $M_1,M_2,M_3,M_4$ be a corresponding set of maximal subgroups in general position, i.e. $H_i\leq M_i$, $i=1,...,4$.  Let $p\equiv\pm 1\mod 10$ or $p\equiv \pm 1\mod 8$.  In the course of their proof, Whiston and Saxl~\cite{W&S} show that in the case $m(G)=4$, it must be that there exists an $i$ such that $H_i\cong\mathrm{S}_4$ or $H_i\cong\mathrm{A}_5$.  In fact, one can learn even more in general about the $g_i$ and the $H_i$.   Another proposition in Whiston and Saxl's paper~\cite{W&S} says the following:

\begin{prop}
\label{prop:twoS}
No more than three $H_i$ can be of the form $\mathrm{D}_{p\pm 1}$ or $G_p$.  If three of the $H_i$ are of this form, then $m(G)=3$.  
\end{prop}

This means that when $m(G)=4$ at least {\it two} of the $H_i$ must be isomorphic to $\mathrm{A}_5$ or $\mathrm{S}_4$.  To proceed, it is important to understand the generating sequences of $\mathrm{S}_4$ and $\mathrm{A}_5$.   First of all, from Whiston's thesis~\cite{WPhD}, $m(\mathrm{S}_n)=n-1$ which is $3$ for $\mathrm{S}_4$ and since $\mathrm{A}_5\cong \mathrm{PSL(2,5)}$, $m(\mathrm{A}_5)=3$.  Next, note the following.

\begin{lemma}
\label{prop:mustgen}
Every irredundant sequence of length $3$ in $\mathrm{S}_4$ or $\mathrm{A}_5$ must generate.  As was remarked earlier, $\mathrm{S_4}$~and $\mathrm{A_5}$~are strongly flat.
\end{lemma}

\begin{proof}
This follows from a careful consideration of the lattice of subgroups.  The union of the sets of possible subgroups for these two groups have isomorphism classes $\{\mathrm{A}_4,\mathrm{D}_{10},\mathrm{D}_8,\mathrm{S}_3,\mathbb{Z}_5,\mathbb{Z}_2^2,\mathbb{Z}_4,\mathbb{Z}_2,\{e\}\}$.  All of these groups have $m(H)\leq 2$.
\end{proof}

Since two of the $H_i$ must be isomorphic to $\mathrm{S}_4$ or $\mathrm{A}_5$, without loss of generality, suppose that $H_1$ and $H_2$ satisfy this condition.  From the maximality of $\mathrm{S}_4$ and $\mathrm{A}_5$, we can further deduce that $M_1\cong H_1$ and $M_2\cong H_2$.  The only possibilities for $M_3$ and $M_4$ by Dickson's Theorem are $\mathrm{S}_4,D_{p\pm 1}$ and $G_p$.  In fact, for length four sequences, this last subgroup is not possible.  

\begin{lemma}
\label{funnysubgroup}
Suppose that $p\equiv \pm 1\mod 10$ or $p\equiv\pm 1\mod 8$.   Suppose that $m(G)=4$.  Then, $H_3$ and $H_4$ are not isomorphic to a subgroup of $G_p$.
\end{lemma}

\begin{proof}
Suppose on the contrary that $M_4\cong G_p$.  The subgroups $L\leq M_4$ are isomorphic to $\mathbb{Z}_p,C$ or $\mathbb{Z}_p\rtimes C$, where $C$ is a cyclic subgroup of one of the isomorphic copies of $\mathbb{Z}_{(p-1)/2}\leq M_4$.  Since $p>5$, the only type which will have a potentially nontrivial intersection with $H_1$ is the cyclic subgroups.  Thus, to be in general position, $H_1\cap H_4$ must be cyclic.  The only cyclic subgroups of $\mathrm{S}_4$ and $\mathrm{A}_5$ have order $2,3$ or $4$, but $m(\mathbb{Z}_2)=m(\mathbb{Z}_3)=m(\mathbb{Z}_4)=1$.  Therefore, $|H_1\cap H_4\cap H_2|=1$.  This contradicts the fact that these subgroups are in general position.
\end{proof}

Now, we can now begin to quantify what is meant by `gluing' sequences.  Since $H_1$ and $H_2$ are isomorphic to either $\mathrm{S}_4$ or $\mathrm{A}_5$, every length four irredundant generating sequence in $G$ is the composite of two length three irredundant generating sequences from $\mathrm{S_4}$~ or $\mathrm{A_5}$.  From this fact, it is clear that the next step is to study the length $3$ irredundant generating sets of $\mathrm{S_4}$~ and $\mathrm{A_5}$.   In their paper~\cite{CC}, Cameron and Cara determine all the length $n-1$ irredundant generating sets of $\mathrm{S}_n$ except when $n=4$ and $6$.  As they suggest, we approach $n=4$ with a computation using GAP~\cite{GAP}, which reveals that elements in length three irredundant (generating) sequences have order $2$ or $3$. 

Now, we turn our attention back to $H_1$ and $H_2$;  $g_2,g_3,g_4$ is an irredundant generating sequence of length $3$ in $\mathrm{S_4}$~ or $\mathrm{A_5}$.  Thus, $g_2,g_3,g_4$ have orders $2$ or $3$.  Repeating this same argument for $g_1,g_3,g_4$ reveals that $g_1$ also must have order $2$ or $3$.  Therefore, 

\begin{prop}
\label{prop:orders}
If $m(G)=4$, the possible orders of elements in an irredundant generating sequence of length $4$ in $\mathrm{PSL}(2,p)$~are $2$ and $3$.  
\end{prop}

Now, we will specialize to the case $p\not\equiv\pm 1\mod 10$ and begin the proof of Theorem~\ref{thm:main}.  First, we consider a special case of Prop.~\ref{prop:orders}.

\begin{cor}
If $p\equiv\pm 1\mod 8$ but $p\not\equiv\pm 1\mod 10$ and $m(G)=4$ then all the elements of an irredundant generating sequence of maximal length have order $2$.
\label{allorder2}
\end{cor}

\begin{proof}
First, note that $\langle g_i\rangle\leq (M_1\cap M_j)\cap (M_1\cap M_k)$, where $1,i,j,k$ are all different.  The only way for $g_i$ to have order $3$ is for $3$ to divide the orders of both $L_j=M_1\cap M_j$ and $L_k=M_1\cap M_k$.  The only subgroups of $\mathrm{S}_4$ with this property are isomorphic to $\mathbb{Z}_3,\mathrm{S}_3$ or $\mathrm{A}_4$.  The intersections $L_j$ and $L_k$ cannot be cyclic of prime order because then the $H_i$ will not be in general position (the intersection of three will be trivial).  First, suppose that both $L_j$ and $L_k$ are isomorphic to $\mathrm{S}_3$.  Further suppose that $g_i$ has order $3$ and $g_i\in L_j\cap L_k$.  The subgroup generated by $g_i$ is normal in $L_j$ and $L_k$.  However, since $\mathrm{S}_3$ is maximal in $\mathrm{S}_4$, the normalizer in $\mathrm{S}_4$ of $\langle g_i\rangle$ is $\mathrm{S}_3$, i.e. there is a unique $\mathrm{S}_3$ which contains $\langle g_i\rangle$.  This contradicts the fact that both $L_j$ and $L_k$ contain $\langle g_i\rangle$.  We cannot have the intersection of two copies of $\mathrm{A}_4$ since a given $\mathrm{S}_4$ has only one of these subgroups.\\
\\All that remains is to show that one cannot have the intersection of an $\mathrm{S}_3$ and a $\mathrm{A}_4$.  In order for one of $L_j,L_k$ to be $\mathrm{A}_4$, it must be that one of $H_j,H_k$ is $\mathrm{S}_4$, since this is the only subgroup of $\mathrm{PSL}(2,p)$ which could contain an $\mathrm{A}_4$ (it is not cyclic or dihedral).   Therefore, we can apply the same argument as we used for two copies of $\mathrm{S}_3$.  In particular, $\mathrm{A}_4$ is normal in $\mathrm{S}_4$, which is maximal in $\mathrm{PSL}(2,p)$.  Thus, there is a unique $\mathrm{S}_4$ which contains the $\mathrm{A}_4$, a contradiction.  Thus, by Cor.~\ref{prop:orders}, $g_2,g_3,g_4$ have order $2$.  Clearly, we could have switched $g_1$ and $g_2$, which shows that $g_1$ also has order $2$.

\end{proof}

The general strategy for combining generating sequences and proving Theorem 3 is now as follows.  Generically, consider 

\begin{align}
\label{biggroup}
Q_R=\langle x_1,x_2,x_3,x_4 | x_i^2=1,R\rangle,
\end{align}

\noindent where $R$ is a set of relations.  By Cor.~\ref{allorder2}, when $R=\emptyset$ and $m(G)=4$, $G$ is a quotient of $Q_R$.  In particular, we have an explicit map which sends $x_i\mapsto g_i$.  The strategy now is to make $R$ as big as possible.  Quantitatively, we choose $R$ to contain all the information we know about the generating sequences of $\mathrm{S}_4$ and the ways in which dihedral groups can intersect with $\mathrm{S}_4$ and each other.  First, we consider how to include information about generating sequences of $\mathrm{S}_4$.  Let $s=(s_1,s_2,s_3)$ be an irredundant generating sequence of $\mathrm{S}_4$.  Let

\begin{align}
\Lambda_n^s=\left\{(\alpha,\beta)\in (1,2,3)^n\times\mathbb{Z}^n\Bigg|\prod_{i=1}^n s^{\beta_i}_{\alpha_i}=1\right\}.
\end{align}

\noindent For $\lambda\in\Lambda_n^s$, define $f_\lambda:\{\text{three letters}\}\rightarrow\{\text{words from three letters}\}$: 

\begin{align}
f_\lambda(x_1,x_2,x_3)=\prod_{i=1}^n x^{\beta_i}_{\alpha_i}.
\end{align}

\noindent Now, for a given $s$, we construct an object $R_s(x_1,x_2,x_3)$ which carries all the information in $s$ needed to build $\mathrm{S}_4$ from a free group such that there is a sequence isomorphic to $s$ as an irredundant generating sequence.  More precisely, let

\begin{align}
\tilde{R}_s(x_1,x_2,x_3)=\bigcup_{n\in\mathbb{N}}\bigcup_{\lambda\in\Lambda_n^s}\{f_\lambda (x_1,x_2,x_3)=1\}.
\end{align}

\noindent By construction $\langle x_1,x_2,x_3|\tilde{R}_s(x_1,x_2,x_3)\rangle \cong \mathrm{S}_4$ and the image of $x_1,x_2,x_3$ under the canonical map is isomorphic to $s$.  In practice, one picks $R_s\subseteq \tilde{R}_s$ such that $|R_s|<\infty$.  An $R_s$ exists because $\mathrm{S}_4$ is finite and a group is uniquely defined by its complete multiplication table.  Thus, for example, one can construct $R_s$ by enumerating the $24$ elements of $\mathrm{S}_4$ in words of elements of $s$ and then encoding the $24\times 24$ multiplication table in terms of relations.  For example, write $\mathrm{S}_4=\{w_i(s_1,s_2,s_3),i=1,...,24\}$ for $w_i$ some functions that send $s_1,s_2,s_3$ to elements of $\mathrm{S}_4$ built from these generators.   If $w_i(s)w_j(s)=w_k(s)$, then one would include in $R_s$ the term $w_i(x)w_j(x)w_k(x)^{-1}=1$ for $x=(x_1,x_2,x_3)$.  One can construct an $R_s$ explicitly in GAP and often it is possible to choose $R_s$ with a size much smaller than $24^2$.  For instance, if $s=((23),(14),(12))$, then for 

\begin{align*}
R_s(x_1,x_2,x_3)=\{x_i^2&=(x_1x_2)^2=(x_2x_3)^3=(x_1x_3)^3=(x_1x_2x_3)^4\\
&=(x_1x_2x_3x_2)^3=1\},
\end{align*}

\noindent we get that $\langle x_1,x_2,x_3|R_s(x_1,x_2,x_3)\rangle \cong \mathrm{S}_4$ and the image of $x_1,x_2,x_3$ under the canonical map is isomorphic to $s$.

Now, we return to the task of considering $Q_R$ from Eq.~\ref{biggroup}.   There will be three cases, depending on the group type of $M_3$ and $M_4$.  Let $S$ be the set of length $3$ irredundant generating sequences of $\mathrm{S}_4$.  One can easily compute $S$ from GAP.  

\begin{enumerate}

\item $M_3\cong M_4\cong\mathrm{S}_4$. For $s_1,s_2,s_3,s_4\in S$, we consider 

\begin{align*}
Q=\langle x_1,x_2,x_3,x_4 | &R_{s_1}(x_1,x_2,x_3),R_{s_2}(x_1,x_2,x_4),\\
&R_{s_3}(x_1,x_3,x_4),R_{s_4}(x_2,x_3,x_4)\rangle.
\end{align*}

A general strategy for determining if a finitely presented group is finite is to use the {\it Todd-Coxeter} algorithm~\cite{TC}.   For example, consider the case $s_1=s_2=s_3=s_4=s=((23),(14),(12))$ from above.   Then, we find -- using the GAP implementation of the Todd-Coxeter algorithm -- that $|Q|=6$ and thus $Q$ is too small to have $G$ as a quotient.  Repeating this calculation for all sets of four elements of $S$, we find that either $|Q|\leq 192$ or the Todd-Coxeter algorithm does not terminate in a reasonable amount of time.  In all the latter cases, there exists $R'\subseteq R$ with $R'=\{(x_ix_j)^{m_{ij}}=1\}$ where

\begin{align*}
m=\begin{pmatrix}2&3&2&3\cr 3&2&3&2\cr 2&3&2&3\cr 3&2&3&2\end{pmatrix},
\end{align*}

and thus $Q$ is a quotient of the Coxeter group $Q'=\langle x_1,x_2,x_3,x_4|R'\rangle$.  This Coxeter group is well known -- $Q' \cong \tilde{A}_3\cong \mathbb{Z}^3\rtimes \mathrm{S}_4$~\cite{KB} which is solvable and thus cannot have the simple group $G$ as a quotient.   In the case $|Q|<192$,  $Q$ is only big enough to have $G$ as a quotient if $p=7$.  A direct computation shows that $G=\mathrm{PSL}(2,7)$ does in fact have $m(G)=4$.

\vspace{3mm}

\item Without loss of generality, $M_3\cong \mathrm{S}_4$ and $M_4$ is dihedral.  We lose one constraint $R_{s_1}(x_1,x_2,x_3)$ and so we will need additional information, from the intersection of dihedral groups and $\mathrm{S}_4$:

\begin{lemma}
Let $K_1=H_1\cap H_4,K_2=H_2\cap H_4$ and $K_3=H_3\cap H_4$.   If $H_3\cong\mathrm{S}_4$ and $H_4$ is dihedral, then no two of the $K_i$ can be isomorphic to $\mathrm{S}_3$.
\end{lemma}

\begin{proof}
First, we note that $H_4$ has a unique cyclic subgroup $L$ of order $3$ since it is dihedral.  Suppose that $K_1\cong K_2\cong \mathrm{S}_3$.  Since the $H_i$ are in general position, $K_1\neq K_2$ and so $L=H_1\cap H_2\cap H_4$.   However, $\langle g_3\rangle \leq H_1\cap H_2\cap H_4$ which means that $g_3$ has order dividing $3$, a contradiction.   
\end{proof}

\begin{cor}
\label{cor:notwo}
No two of $Order(g_1g_2),Order(g_1g_3),Order(g_2g_3)$ can be $3$
\end{cor}

\begin{proof}
By the lemma, no two of the $K_i$ can be isomorphic to $\mathrm{S}_3$.  This means no two of $\{\langle g_1,g_2\rangle,\langle g_1,g_3\rangle$,$\langle g_2,g_3\rangle\}$ can be isomorphic to $\mathrm{S}_3$.  All the $g_i$ have order $2$, so no two of $\{Order(g_1g_2)$, $Order(g_1g_3)$, $Order(g_2g_3)\}$ can be $3$.  
\end{proof}

A similar result is true for $\mathrm{D_8}$.

\begin{lemma}
No two of $Order(g_1g_2),Order(g_1g_3)$ and $Order(g_2g_3)$ can be 4.
\end{lemma}

\begin{proof}
Suppose without loss of generality that $Order(g_1g_2)=Order(g_1g_3)=4$.  Then, $H_3\cap H_4\cong H_2\cap H_4\cong \mathrm D_8$.  To see this, note for example that $\langle g_1,g_2\rangle\leq H_3\cap H_4$, but $\langle g_1,g_2\rangle\cong D_8$, which is maximal and so this is equality.  Next, note that $H_4$ has a unique cyclic group of order $4$ which is in common to both of $H_3\cap H_4$ and $H_2\cap H_4$.  Therefore,  $H_2\cap H_3\cap H_4$ is cyclic of order $4$ (it cannot be all of $D_8$ since then the $H_i$ would not be in general position).  In $H_2\cap H_4$, the cyclic group of order four is $\langle g_1g_3\rangle$ and in $H_3\cap H_4$, the cyclic group of order four is $\langle g_1 g_2\rangle$.  The fact that these are the same means $g_1g_2=g_1g_3$ or $g_1g_2=(g_1g_3)^3$.  Then, we can write $g_2=g_1(g_1g_3)^x$, where $x=1$ or $3$.   This contradicts the irredundantcy of the $g_i$.  

\end{proof}

Imposing the conditions $R_{s_2}(x_1,x_2,x_4),R_{s_3}(x_1,x_3,x_4)$, and $R_{s_4}(x_2,x_3,x_4)$ alongside those in the previous two lemmas, the Todd-Coxeter algorithm gives $|Q|<1344$.   After $p=7$, the next prime $\pm 1\mod 8$ is $p=17$, but $|\mathrm{PSL}(2,17)|=2448$.  A direct computation shows that in fact, all the length four irredundant generating sequences of $\mathrm{PSL}(2,7)$ correspond to $M_i\cong \mathrm{S}_4$ for all $i=1,...,4$ and so this case cannot occur.

\vspace{3mm}

\item $M_3$ and $M_4$ are dihedral.  We only have $R_{s_3}(x_1,x_3,x_4)$ and $R_{s_4}(x_2,x_3,x_4)$ by requiring $M_1$ and $M_2$ to be isomorphic to $\mathrm{S}_4$.  Thus, we need further constraints from the following lemma: 

\begin{lemma}
If $M_3$ and $M_4$ are dihedral groups, then $M_1\cap M_2\cong\mathbb{Z}_2^2$.
\end{lemma}

\begin{proof}
Since $M_3$ and $M_4$ are dihedral, $H=M_3\cap M_4$ must be cyclic or dihedral.   Suppose that $H$ is cyclic and let $K$ be an index two cyclic subgroup of $M_3$.  Because the $M_i$ are in general position, $H$ cannot have order $2$.  Therefore, $H\unlhd K$.   However, every subgroup of a cyclic subgroup is characteristic and so $H\unlhd M_3$.  Since $M_3$ is maximal in $G$, it must be that $M_3=N_G(H)$.  However, the same argument shows that $M_4=N_G(H)$.  Therefore, $M_3=M_4$, a contradiction.  Therefore, $H$ must be dihedral.  Let $L\leq H$ be the cyclic subgroup of index $2$.  By our previous discussion, if $|L|>2$, there would be a unique dihedral group in $G$ which contains $L$, which is a contradiction.  Therefore, we must have $|L|=2$ and so $H\cong\mathbb{Z}_2^2$. 
\end{proof}

Imposing the conditions $R_{s_3}(x_1,x_3,x_4)$ and $R_{s_4}(x_2,x_3,x_4)$ alongside the constraint form the above lemma gives two outcomes.  When the Todd-Coxeter algorithm terminates in a reasonable amount of time, $|Q|<1344$ which has already been ruled out.  There are two configurations of $Q$ for which the Todd-Coxeter algorithm does not terminate in a reasonable amount of time.   In one case, $Q$ is a quotient of $\tilde{A}_3$, which we have already discussed is not possible by solvability (Coxeter diagram on the left below).  In the second case, $Q$ is a quotient of the Coxeter group represented by the diagram on the right, below:

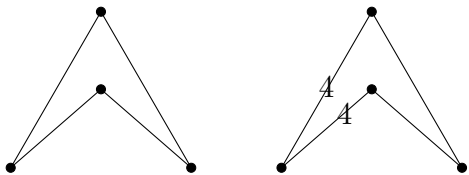
\begin{center}
\begin{tikzpicture}[scale=1.2]
	\draw (0,0.87) -- (-1,0);
	\draw (1,0) -- (0,0.87);	
	\draw (-1,0) -- (0,1.73);
	\draw (1,0) -- (0,1.73);	
	\draw [fill] (0,1.73) circle [radius=0.05];
	\draw [fill] (-1,0) circle [radius=0.05];
	\draw [fill] (0,0.87) circle [radius=0.05];
	\draw [fill] (1,0) circle [radius=0.05];
	\begin{scope}[shift={(3,0)}]
	\draw (0,0.87) -- (-1,0);
	\draw (1,0) -- (0,0.87);	
	\draw (-1,0) -- (0,1.73);
	\draw (1,0) -- (0,1.73);	
	\draw [fill] (0,1.73) circle [radius=0.05];
	\draw [fill] (-1,0) circle [radius=0.05];
	\draw [fill] (0,0.87) circle [radius=0.05];
	\draw [fill] (1,0) circle [radius=0.05];
	\node at (-0.3,0.6) {$4$};
	\node at (-0.5,0.9) {$4$};
\end{scope}
\end{tikzpicture}
\end{center}

Let $C$ denote the Coxteter group corresponding to the diagram on the right.  Even though the Todd-Coxeter algorithm does not terminate, it can be used to determine subgroups of finite index.  One finds that $C$ has a subgroup $C'\unlhd C$ such that $C/C'\cong\mathrm{S}_4$.  Furthermore, a straightforward application of the relations shows that $C'$ is generated by three elements which mutually commute.  Since $C'$ is abelian and $C/C'$ is solvable, $C$ is also solvable and thus $G$ cannot be a quotient.

\end{enumerate}

\section{$\mathrm{PSL}(2,p)$~and the Replacement Property}

While it is not known in general if $G=\mathrm{PSL}(2,p)$ satisfies the replacement property, in some special cases, we can say whether $G$ has this property or not.    

\begin{thm}
(R. K. Dennis) Let $G$ be a finite group, $m=m(G)$ and $s=(g_1,...,g_m)$ is an irredundant generating sequence of length $m$.  Let $F=\{M_1,...,M_m\}$ be an associated family of maximal subgroups in general position.  Assume that for any such $F$, there exists one of the maximal subgroups, say $M_m$ such that
\begin{enumerate}
	\item $M_m=\langle g_1,...,g_{m-1}\rangle$
	\item $m(M_m)=m-1$
	\item $M_m$ satisfies the replacement property.
\end{enumerate}
Then, $G$ satisfies the replacement property
\end{thm}

\begin{proof}
Note that for $j\neq m$ we have $M_m\cap M_j\neq M_m$ since $F$ is in general position.  Thus, there exists $N_j\in \mathrm{Max}(M_m)$ (the set of maximal subgroups of $M_m$) with $N_j\geq M_m\cap M_j$. Hence, $F'=\{N_1,...,N_{m-1}\}$ is a family of maximal subgroups of $M_m$ in general position associated to the irredundant generating sequence $s'=(g_1,...,g_{m-1})$, since $M_m\cap M_j\geq \langle s(\hat{m},\hat{j})\rangle$ (the sequence generated by all the $g_i$ for $i$ not $m$ and not $j$).  Since $M_m$ satisfies the replacement property, we have that $N_1\cap\cdots\cap N_{m-1}$ is trivial.  Thus, $M_1\cap \cdots \cap M_m=(M_m\cap M_1)\cap\cdots\cap (M_m\cap M_{m-1})\leq N_1\cap \cdots\cap N_{m-1}=\{e\}$.  Therefore, $G$ satisfies the replacement property.

\end{proof}

\begin{cor}
Let $G=\mathrm{PSL}(2,p)$, $p$ prime and $m(G)=4$.  Then, $G$ satisfies the replacement property.
\end{cor}

\begin{proof}
We know that every $F$ must contain a group isomorphic to either $\mathrm{S}_4$ or $\mathrm{A}_5$.   By Lemma~\ref{prop:mustgen}, length 3 irredundant sequences must generate.  Both $\mathrm{S}_4$ and $\mathrm{A}_5$ satisfy the replacement property, so $G$ does as well by the theorem.
\end{proof}

It turns out that the theorem can also be applied to $M_{11}$, the sporadic group since $\mathrm{PSL(2,7)}$ is in every $F$, $m(M_{11})=5$, and we just showed that $\mathrm{PSL(2,7)}$ satisfied the replacement property.   However, $\mathrm{PSL}(2,p)$ does not satisfy the replacement property in general:

\begin{thm}
Let p be a prime with $p\equiv +1 \mod 8$.  Let $G=\mathrm{PSL}(2,p)$.  If $m(G)=3$, then $G$ fails the replacement property.  
\label{thesistheorem}
\end{thm}

\begin{proof}
In order to show that $G$ fails the replacement property, this proof produces an explicit example of an element $w\in G$ and a length three generating set $\{g_1,g_2,g_3\}$ such that replacing any $g_i$ by $w$ will result in a set which no longer generates $G$.   Since it is easier to work with matrices than with elements in $\mathrm{PSL}(2,p)$, often, elements in $\mathrm{SL}(2,p)$ will be used instead of their projections into $G$.  For the sake of clarity, capital letters will denote elements in $\mathrm{SL}(2,p)$ and lower case letters will denote their projections in $G$.  

\vspace{5mm}

Let $a,b,c,w\in G$ and for the canonical projection, $\pi:$ $\mathrm{SL}(2,p)$$\rightarrow G$,  let $\pi(A)=a,\pi(B)=b, \pi(C)=c$ and $\pi(W)=w$.  We will construct $a,b,c,w$ such that $\{wa,wb,wc\}$ is a length $3$ irredundant generating set of $G$, but the element $w$ will be such that it cannot replace any of these elements to recover a generating sequence.  For $r,s,t,u\in\mathbb{F}_p$ let

\begin{align}
A=\begin{pmatrix}r & s\cr s & -r\end{pmatrix}\hspace{5mm}B=\begin{pmatrix}t & u\cr u & -t\end{pmatrix}\hspace{5mm}W=\begin{pmatrix}0 & -1\cr 1 & 0\end{pmatrix}.
\end{align}

Since $A$ and $B$ have determinant $1$, $r^2+s^2=t^2+u^2=-1$.  Note that $A,B$ and $W$ are traceless.  A standard result~\cite{MS} then says that $A,B$ and $W$ have order $4$ and $a,b$ and $w$ have order $2$.  Furthermore, notice that

\begin{align}
WA=\begin{pmatrix}-s & r\cr r & s\end{pmatrix}\hspace{5mm}WB=\begin{pmatrix}-u & t\cr t & u\end{pmatrix}\hspace{5mm}AW=\begin{pmatrix}s & -r\cr -r & -s\end{pmatrix}\hspace{5mm}BW=\begin{pmatrix}u & -t\cr -t & -u\end{pmatrix},
\end{align}

and so $AW=-WA$ and $BW=-WB$.  Since $AW,AB$ are still traceless, $aw$ and $bw$ also have order $2$.  Therefore, $\langle a,w\rangle = \{a,w,aw,\mathrm{id}\}\cong\mathbb{Z}_2\times\mathbb{Z}_2$ and likewise, $\langle b,w\rangle\cong \mathbb{Z}_2\times\mathbb{Z}_2$.   Now, generically write 

\begin{align}
C=\begin{pmatrix}\alpha & \beta\cr \gamma & \delta\end{pmatrix},
\end{align}

where $\alpha\delta-\beta\gamma=1$.  Let $\alpha+\delta=0$ so that $c$ has order $2$.  Furthermore, $\mathrm{Tr}(WC)=+1$ which means that the order of $wc$ is $3$~\cite{MS}.  Note that

\begin{align}
WC=\begin{pmatrix}-\gamma & -\delta\cr \alpha& \beta\end{pmatrix},
\end{align}

and so the condition $\mathrm{Tr}(WC)=+1$ becomes $\beta-\gamma=+1$.  Choose $\beta=0$ so that $\gamma=-1$.  Furthermore since $\alpha\delta-\beta\gamma=1$, $\beta=0$ implies that $\alpha=\delta^{-1}$ and since the trace of $C$ is zero, $\alpha=-\delta$.  Thus, $\alpha^{-1}=-\alpha$, or $\alpha$ has order $4$ in $\mathbb{F}_p$.  Does such an element exist?  Since $p\equiv 1\mod 8$, $8|(p-1)$, which is the order of the cyclic group $\mathbb{F}_p^*$.  Therefore, $\mathbb{F}_p^*$ has an element of order $8$ and so also has an element of order $4$.  Fix such an element and call it $i$.  Then, 

\begin{align}
C=\begin{pmatrix}-i & 0\cr -1 & i\end{pmatrix}.
\end{align}


Note that $w(cw)w^{-1}=wcww=wc$.  However, since $(cw)(wc)=1$, $w(cw)w^{-1}=(cw)^{-1}$.  Therefore, $\langle c,w\rangle=\langle w,wc\rangle=\langle x,y|x^2=y^3=1,xyx^{-1}=y^{-1}\rangle\cong \mathrm{S}_3$.   The next step is to show that $\langle aw,cw\rangle\cong\mathrm{S}_4$.  The idea is to use the trace technology laid out in~\cite{DM}.  The trace of $WA$ is $0$ and the trace of $WC=+1$.  The Main Theorem in~\cite{DM} requires that $WCWA$ has a particular trace.  Multiplying these elements gives rise to the following matrix:

\begin{align}
WCWA=\begin{pmatrix}-s-i r&r-i s\cr i s&-i r\end{pmatrix},
\end{align}

so that $\mathrm{tr}(WCWA)=-s-2i r$.  The required constraint from the Theorem is that $(s+2ir)^2=2$.  If this holds, then $\langle aw,cw\rangle\cong\mathrm{S}_4$ if $\mathrm{tr}\left([WA,WC]\right)=+1$.  Simple arithmetic using the forms of $A,C$ and $W$ shows that $\mathrm{tr}\left([WA,WC]\right)=2 s^2+4i s r-3 r^2$.  Setting this expression equal to $1$ and using the constraint that $s^2+r^2=-1$ (from the determinant), one finds that  $3s^2+4isr-2r^2=0$,  which has solution $r=\left(i\pm 1/\sqrt{2}\right)s$.  Inserting this back into $s^2+r^2=-1$ yields


\begin{align}
\label{eq:answer}
s^2=-\frac{2}{9}\pm\frac{4}{9}i\sqrt{2}=\left[\frac{1}{3}\left(2i\pm\sqrt{2}\right)\right]^2,
\end{align}

and so the question has simply boiled down to the existence of an element $\zeta\in\mathbb{F}_p$ such that $\zeta^2=2$ (and $p\neq 3$, so $3^{-1}$ makes sense).   It is a standard result in elementary number theory (c.f.~\cite{KZ}) that $2$ has a square root if $p\equiv \pm 1\mod 8$ (fix one and call it $\sqrt{2}$).   Using the expressions for $r$ and $s$ above, a quick computation shows that $-s-2i r=\sqrt{2}$, as required by the theorem.  Therefore, $\langle wa,wc\rangle\cong\mathrm{S}_4$.  An analogous discussion shows that if one fixes $s$ as one solution to Eq.~\ref{eq:answer}, then picking the other solution for $u$ and constructing $t$ as was done for $r$ will give $\langle wb,wc\rangle\cong\mathrm{S}_4$ as well.  

\vspace{5mm}

The strategy to demonstrate that $w$ cannot replace $wa,wb$ or $wc$ will be to show that $w$ is in the subgroups generated by (maximal subgroups containing) $\langle wa,wc\rangle$, $\langle wb,wc\rangle$ and $\langle wa,wb\rangle$.  The first step in this process is to prove that $\langle wa,wc\rangle=\langle a,c,w\rangle$.  Note that $WAWC=-AWWC=AC$, and since $(ac)(ca)=1$, $ac,ca\in \langle wa,wc\rangle$.  Furthermore, since $(wc)(cw)=(aw)(wa)=1$, $cw,wc,aw,wa\in \langle wa,wc\rangle$.  Now, take any element $x\in\langle a,c,w\rangle$.  By construction, such an element can be written as a string in the alphabet $a,c,w,a^{-1}=a,c^{-1}=c,w^{-1}=w$ (no need to worry about uniqueness).  Suppose that $x$ can be written with an even number of letters in the string.  Then, this element is in $\langle wa,wc\rangle$ because every possible pairing of letters from the above alphabet is in $\langle wa,wc\rangle$.  
 
 Instead of an even number of letters, suppose that $x$ can be written as a string with an odd number of letters from the alphabet.  Then, one can form $x$ from a string in $\langle wa,wc\rangle$ by adding one of $a,b,w$.  This is clear because if there are $n$ letters that make up $x$, then $n-1$ will be an even number and so the substring of the first $n-1$ letters will be in $\langle wa,wc\rangle$ by the preceding argument.   Thus, every element in $\langle a,c,w\rangle$ can be formed from an element in $\langle wa,wc\rangle$ by adding one of $a,c,w$ or $\mathrm{id}$.  This means that $|\langle a,c,w\rangle|\leq 4|\langle wa,wc\rangle|$.

However, from above, $\langle wa,wc\rangle\cong \mathrm{S}_4$ so $|\langle a,c,w\rangle|\leq 96$.  Furthermore, by Dickson's Theorem, $\mathrm{S}_4$ is maximal in $G$ and so no proper subgroup can contain $\langle wa,wc\rangle$.  Therefore, either $\langle a,c,w\rangle=\langle wa,wc\rangle$ or $\langle a,c,w\rangle=G$.  Since $p\equiv 1\mod 8$, $p\geq 17$ so $|G|\geq 2448>96$ and thus $\langle a,c,w\rangle=\langle wa,wc\rangle$.  By an analogous argument, $\langle b,c,w\rangle=\langle wb,wc\rangle$.  The last consideration is to study $\langle wa,wb\rangle$.  This group is generated by two elements of order $2$ and so must be dihedral.  To see how large it is, one needs to know the order of $wawb=awwb=ab$.  This amounts to computing the trace of $AB$, which is

\begin{align}
\mathrm{tr}(AB)=2(rt+su)=-8i/3.
\end{align}

This is certainly not zero and a quick arithmetic computation shows that it is also not $\pm 1$ or $\pm\sqrt{2}$.  Therefore, from a characterization of element orders based on traces, the order of $ab$ is more than $4$ and so $ab\not\in\mathrm{S}_4$.  It is also clear that $\langle wa,wb\rangle\neq G$ because $G$ is not dihedral.  The final step before concluding is to show that $\langle a,b,w\rangle$ is a proper subgroup of $G$.   This procedure is similar to the one above by considering the index of $\langle wa,wb\rangle$ in $\langle a,b,w\rangle$.  Since $wawb=ab\in \langle wa,wb\rangle$, as before, every possible pair of letters in $\langle a,b,w\rangle$ is in $\langle wa,wb\rangle$ and therefore, one arrives at the same bound as earlier $|\langle a,b,w\rangle|\leq 4|\langle wa,wb\rangle|$.

Recall that $\langle wa,wb\rangle$ is dihedral.  From Dickson's Theorem, the largest dihedral subgroup of $G$ has order $p+1$.  Therefore  

\begin{align}
|\langle a,b,w\rangle|\leq 4|\langle wa,wb\rangle|\leq 4(p+1)<p(p+1)(p-1)/2.
\end{align}

Since for $p\geq 17$, $p(p-1)/2=136$.  Let $\langle wa,wb\rangle\leq M<G$ be maximal.  Since $\langle a,b,w\rangle$ is proper and contains $\langle wa,wb\rangle$, $w\in M$.  

\vspace{5mm}

Now, all the machinery is in place to conclude.   The set $\{wa,wb,wc\}$ will generate $G$ because $wb\not\in \langle wa,wc\rangle$ and $\langle wa,wc\rangle$ is maximal, so the subgroup generated by all three elements, which contains a maximal subgroup, must be all of $G$.  Furthermore, it is clear that $w$ cannot replace any of $wa,wb,wc$ because $w$ is in the maximal subgroup containing each pair.  Explicitly, the set $\{w,wb,wc\}$ cannot generate $G$ because $w\in \langle wb,wc\rangle\cong\mathrm{S}_4$.  The same holds for replacing $wb$.  Finally, $w$ cannot replace $wc$ because the maximal subgroup which contains $\langle wa,wb\rangle$ also contains $w$ and so $\langle w,wa,wb\rangle\leq M<G$.  Therefore $G$ fails the replacement property if $m(G)=3$.
\end{proof}

\begin{cor}
If $p\not\equiv\pm 1\mod 10$ and $p\neq 7$ but $p\equiv+ 1\mod 8$, then $G$ fails the replacement property.
\end{cor}

\begin{question}
What are all the cases for which $G$ satisfies the replacement property when $m(G)=3$?
\end{question}

\section{Elements of Irredundant Generating Sequences, $\iota_n(G)$}

Define $\iota_n(G)$ to be the set of orders of elements which appear in length $n$ generating sequences.  Clearly, for $n>m(G)$, $\iota_n(G)$ is the empty set.  For $G=\mathrm{PSL}(2,p)$, $\iota_1(G)$ is also the empty set as $G$ is not cyclic.   R. Guralnick~\cite{RG} proved a powerful theorem for length 2 generating sets of simple groups\footnote{This proof invokes the classification of finite simple groups.}:

\begin{thm}[3/2 Generation] Given any $x\in G$, there exists $y\in G$ so that $G=\langle x,y\rangle$.  In particular, $r(G)=2$ for $G$ a non-abelian finite simple group.
\end{thm}

Therefore, by the $3/2$ Generation Theorem, $\iota_2(G)=\{d|\text{$d$ divides $|G|$ and $d$ is not $1$}\}$, i.e. every non-identity element is in a length $2$ irredundant generating sequence.  Thus, all that remains to be determined is $\iota_3(G)$ and $\iota_4(G)$.  For finite vector spaces $\mathbb{F}_p^n$, $p$ prime, it is clear that all elements of a generating sequence of maximal length (basis) have prime order.  For non-abelian simple groups, this is not necessarily true.  We can see this as a result of the following proposition:

\begin{prop}
For $G=$$\mathrm{PSL}(2,p)$, there is always a length $3$ irredundant generating sequence were all three elements have order $(p-1)/2$.
\end{prop}

\begin{proof}
Let $\pi: \mathrm{SL}(2,p)\rightarrow G$ be the canonical projection.  Let $a,b,c$ be elements of $G$ and $A,B,C$ be lifts to matrices in $\mathrm{SL}(2,p)$.  Define:

\begin{align}
A=\begin{pmatrix}x&0\cr 0&\frac{1}{x}\end{pmatrix}\hspace{5mm}B=\begin{pmatrix}\frac{1}{x}&0\cr x&x\end{pmatrix}\hspace{5mm}C=\begin{pmatrix}\frac{1}{x}&y\cr 0&x\end{pmatrix},
\end{align}

where $x\in \mathbb{F}_p^*$ and $y=-x+2/x-1/x^3$.  Note that $A,B$ and $C$ have order $p-1$ and so $a,b$ and $c$ have order $(p-1)/2$.  We claim that $a,b,c$ is the sequence we seek.   First, we note that

\begin{align}
AB=\begin{pmatrix}1&0\cr 1&1\end{pmatrix}\hspace{5mm}AC=\begin{pmatrix}1&xy\cr 0&1\end{pmatrix}\hspace{5mm}BC=\begin{pmatrix}\frac{1}{x^2}&\frac{y}{x}\cr 1&xy+x^2\end{pmatrix},
\end{align}

which all have trace $2$ and thus have order $p$.  It is clear that $A$ is not in $\langle B\rangle \cup\langle C\rangle$, $B$ is not in $\langle A\rangle \cup\langle C\rangle$ and $C$ is not in $\langle A\rangle \cup\langle B\rangle$, since $A$ is diagonal, $B$ is upper triangular and $C$ is lower triangular.  Furthermore, it is clear that $\langle a,b\rangle$, $\langle a,c\rangle$ are not all of $G$ because there will always be a zero in the upper right (lower left) position.  Since each of these groups contain an element of order $p$ and one of order $(p-1)/2$, they are contained in an $H_i$ and thus must exactly generate the $H_i$.  All that remains to show is that $\langle b,c\rangle$ is not all of $G$.   To do this, we will observe that $\langle bc\rangle\unlhd\langle b,c\rangle$.  This will give us the desired result, since $G$ is simple and so has no normal subgroups (and $bc$ has order $p$, so is a proper nontrivial subgroup).  First, a simple computation shows that

\begin{align}
(BC)^n=\begin{pmatrix}\frac{n-(n-1)x^2}{x^2}&\frac{-n(-2x^2+1+x^4)}{x^4}\cr n&\frac{-n+(n+1)x^2}{x^2}\end{pmatrix}.
\end{align}

In order to show that $BC$ is normal, we need to show that conjugating by $B,B^{-1},C,C^{-1}$ takes $BC$ to another power of $BC$.  Given $(BC)^n$ as above, $(BC)^{x^2}=CB$, which means that $CB=C(BC)C^{-1}\in \langle BC\rangle$ and similarly $CB=B^{-1}(BC)B\in \langle BC\rangle$.  Finally, note that

\begin{align}
(BC)^{x^{-2}}=B(CB)B^{-1}=C^{-1}(BC)C.
\end{align}

Therefore, $\langle BC\rangle$ is normal in $\langle B,C\rangle$ since any power of $BC$ conjugates to another power of $BC$ by the generators of $\langle B,C\rangle$.

\end{proof}

\begin{cor}
The elements of length $m(G)$ irredundant generating sequences of $G$ need not have prime order if $G$ is not solvable.  
\end{cor}

\begin{proof}
For $G=\mathrm{PSL(2,13)}$, $m(G)=3$.  By Proposition 22, there exists a length 3 irredundant generating sequence such that all the elements have order $6$.
\end{proof}

\begin{cor}
\label{cor:p-1}
Every divisor of $(p-1)/2$ is in $\iota_3(G)$.
\end{cor}

\begin{proof}
Let $g_1,g_2,g_3$ be a length three irredundant generating set as in the proposition.  Take any $x\in\langle g_1\rangle$ (i.e. an element whose order divides the order of $g_1$, which is $(p-1)/2$).  Since the intersection of all the subgroups $\langle g_i,g_j\rangle$ is trivial, this sequence satisfies the replacement property by Prop.~\ref{important}.  Therefore, $x$ can replace one of the $g_i$ to arrive at a new generating sequence.  Clearly, it can only replace $g_1$.  This new generating set $x,g_2,g_3$ is still irredundant because the set of maximal subgroups in general position associated to the set is the same as it was for the original set of the $g_i$.  
\end{proof}

Now, we consider the elements with order dividing $p+1$ or equal to $p$.  To proceed, we will need the following lemma:

\begin{lemma}
\label{unique}
Let $x\in G=\mathrm{PSL}(2,p)$, $p>5$.  If $x$ has order $p$ or order $>5$ dividing $p+1$, then there is a unique maximal subgroup of $G$ containing $x$.
\end{lemma}

\begin{proof}
First, we note that a dihedral group $D$ of order $2n=p+1$ has a unique cyclic subgroup of order $q$ for every $q>2$ which divides $n$.   This subgroup will be contained in the index two characteristic cyclic subgroup of $D$.  Therefore, the cyclic subgroup $Q$ of order $q$ is normal in $D$.  Since $G$ is simple, $N_G(Q)\neq G$.  Suppose that $Q$ is contained in some maximal subgroup $M\geq N_G(Q)$.  Note that $N_D(Q)=D$ since $Q$ is normal in $D$.  Therefore, $D\leq N_G(D)$, but $D$ is maximal in $G$, so $N_G(D)=D$, i.e. $M=D$; there is a unique maximal subgroup of $G$ which contains $Q$, namely $N_G(Q)$.\\
\\
Now, suppose that $x$ has order $p$.  Since $p>5$, the only maximal subgroup which can contain $\langle x\rangle$ is one isomorphic to $\mathbb{Z}_p\rtimes\mathbb{Z}_{(p-1)/2}$.  However, $\langle x\rangle$ is normal in such a subgroup.  Therefore, the same argument as above applies; there is a unique maximal subgroup which contains $\langle x\rangle$.
\end{proof}

\begin{cor}
\label{cor:p+1}
If $x\in\mathrm{PSL}(2,p)$ has order $>5$ that divides $p+1$ or has order $p$, then $x$ is not in a length $3$ irredundant generating sequence.
\end{cor}

\begin{proof}
Suppose that $\langle g_1,g_2,g_3\rangle$ is an irredundant generating sequence of length $3$ and suppose that $g_1$ has order which divides $p+1$ or has order $p$.  Then, there is a unique maximal subgroup which contains $g_1$.  Let $M_1,M_2,M_3$ be the associated set of maximal subgroups in general position.  Since there is a unique maximal subgroup which contains $g_1$, $M_2=M_3$, which contradicts the fact that they are in general position.  Thus, the sequence of the $g_i$ cannot be irredundant.  
\end{proof}

Note that because of the factor of $2$ in Cor.~\ref{cor:p-1}, we do not necessarily know if $2\in\iota_3(G)$.  From our discussion introducing the replacement property, we observed that $G$ always has an irredundant generating sequence of elements all of order $2$ which has length at least $3$.  If $m(G)=3$, then we immediately get that $2\in\iota_3(G)$.  If $m(G)=4$, then the irredundant generating sequence from the introduction may have length $4$ and then we need additional information to determine if $2\in\iota_3(g)$.  We need a lemma for when $m(G)=4$:

\begin{lemma}
Let $m(G)=4$.  Then, $2\in\iota_3(G)$.
\end{lemma}

\begin{proof}
From the introduction to the replacement property, we know that there exists an irredundant generating sequence with at least $3$ elements all of which have order $2$.  If the length of one such sequence is $3$, then we are done. Instead suppose that we have a length four irredundant generating sequence $g_1,g_2,g_3,g_4\in G$ with $g_i^2=1$.  Without loss of generality, by Whiston and Saxl, we know that at least two of the corresponding maximal subgroups in general position will be isomorphic to $\mathrm{S}_4$ or $\mathrm{A}_5$.  There are two cases:

\begin{enumerate}
\item One of the maximal subgroups is isomorphic to $\mathrm{A}_5$.  Without loss of generality, assume that $\langle g_1,g_2,g_3\rangle\cong \mathrm{A}_5$.  A calculation with GAP shows that all length three irredundant generating sequences of $\mathrm{A}_5$ composed of elements of order 2 have at least one pair whose product has order 5.  Without loss of generality, suppose that the order of $h=g_1g_2$ is $5$.  The sequence $h,g_2,g_3,g_4$ is clearly still a generating sequence of $G$.   However, by Prop. 10, this sequence cannot be irredundant.  By the irredundancy of the original sequence, $(g_2,g_3,g_4)$, $(h,g_2,g_4)$, and $(h,g_2,g_3)$ cannot generate $G$.  Therefore, we must have that $h,g_3,g_4$ generates $G$.  Since $(g_3,g_4)$ and $(h,g_3)$ do not generate $G$, all we need to check in order to prove irredundancy of $h,g_3,g_4$ is that $h$ and $g_4$ do not generate $G$.  If $\langle h,g_4\rangle = G$, then $g_3\in \langle h,g_4\rangle\leq \langle g_1,g_2,g_4\rangle$, which contradicts the irredundancy of the original sequence.
\item None of the maximal subgroups are isomorphic to $\mathrm{A}_5$.  Assume that $\langle g_1,g_2,g_3\rangle\cong \mathrm{S}_4$.  A GAP calculation shows that for all length three irredundant generating sequences composed of elements of order 2, either there is a pair whose product has order 4 or there is a pair whose product has order 3.  In either case, the same logic from the $\mathrm{A}_5$ case applies via Cor. 11 which works even though we have made no assumption about $p$ since there is no $\mathrm{A}_5$ and thus the case is equivalent to $p\not\equiv\pm 1\mod 10$.
\end{enumerate}

\end{proof}

The results about $2\in\iota_3(G)$ mean that Cor.~\ref{cor:p-1} extends to every divisor of $p-1$.   Now, we are ready to summarize what we know about $\iota_n(G)$ using the above results and the properties from Dickson's Theorem:


\begin{thm}
Let $G=\mathrm{PSL}(2,p)$.  Then, for $n>4$, $\iota_1(G)=\iota_n(G)=\emptyset$.  In addition, $\iota_2(G)=\{d|\text{$d$ divides $|G|$ and $d$ is not $1$}\}$.  Furthermore, $\{d|\text{$d$ divides $p-1$}\}=\iota_3(G)$ and $\iota_4=\emptyset$ with the following exceptions:

\begin{enumerate}
\item $p\equiv -1\mod 10$.  Then, $5$ may also be in $\iota_3(G)$ and $\iota_4(G)\subseteq\{2,3\}$.
\item $p\equiv -1\mod 8$.  Then, $4$ may be in $\iota_3(G)$.  If $p=7$ then $\iota_4(G)=\{2\}$. 
\item $p\equiv -1\mod 3$ and $\left( p\equiv 3,13,27,37\mod 40,p\equiv\pm 1\mod 8\text{ or }p\equiv\pm 1\mod 10\right)$.  Then, $3$ may also be in $\iota_3(G)$.
\end{enumerate}

\end{thm}

\begin{question}
What are $\iota_3(G)\cap\{3,4,5\}$ and $\iota_4(G)$ in general?
\end{question}

Answering this may require discovering new methods.  For example, the proof may be achievable with a variation on Hall's 1936 paper~\cite{Hall} which gives the lattice of subgroups, Moebius function, and a formula for $\phi_n(G)$.  Without additional information, the following conjecture about $\iota_4(G)$ might be true:

\begin{conjecture}
For $G=\mathrm{PSL}(2,p)$, $\iota_4(G)=\emptyset$ unless $p=7,11,19,$ or $31$.  In these exceptional cases, $\iota_4(G)=\{2\}$ unless $p=11$ in which case $\iota_4(G)=\{2,3\}$.
\end{conjecture}

The fact that $m(G)=4$ only in these four exceptional cases was verified computationally for primes up to 300~\cite{NachmanThesis} and in the exceptional cases, some important properties of the length four irredundant generating sequences have been computed and are shown in Tables 1 and 2.  Determining $m(G)$ for the final case of $p\equiv -1\mod 10$ will be very interesting as it either confirms or denies the surprising finite list of cases for length four irredundant generating sequences in $\mathrm{PSL}(2,p)$.

\begin{table}[h!]
\begin{center}
\begin{tabular}{ |c|c|c| }
\hline
  & $p=7$ & $p=11$ \\
  \hline
   Length $4$ irredundant generating sets & 252&11935 \\
   Conjugacy classes of sets& 2&22\\
   Automorphism classes of sets& 2&14\\
   Possible Orders of Elements &2&2,3\\
   Families of Maximal Subgroups &\{$\mathrm{S}_4,\mathrm{S}_4,\mathrm{S}_4,\mathrm{S}_4\}$&$\{\mathrm{A}_5,\mathrm{A}_5,\mathrm{A}_5,\mathrm{A}_{5}\}$\\
   &&$\{\mathrm{A}_5,\mathrm{A}_5,\mathrm{A}_5,\mathrm{D}_{12}\}$\\
   &&$\{\mathrm{A}_5,\mathrm{A}_5,\mathrm{D}_{12},\mathrm{D}_{12}\}$\\
  \hline
\end{tabular}
\end{center}
  \label{tab:res11}
  \caption{Properties of length four irredundant generating sequences of $\mathrm{PSL}(2,p)$ for $p=7$ and $p=11$.}
  \end{table}
  
\begin{table}[h!]
\begin{center}
\begin{tabular}{ |c|c|c|}
\hline
  &  $p=19$ & $p=31$ \\
  \hline
   Length $4$ irredundant generating sets &7695&14880 \\
   Conjugacy classes of sets& 4&1 \\
   Automorphism classes of sets& 3&1 \\
   Possible Orders of Elements &2&2\\
   Families of Maximal Subgroups &$\{\mathrm{A}_5,\mathrm{A}_5,\mathrm{A}_5,\mathrm{A}_{5}\}$&$\{\mathrm{S}_4,\mathrm{S}_4,\mathrm{A}_5,\mathrm{A}_5\}$\\
   &$\{\mathrm{A}_5,\mathrm{A}_5,\mathrm{A}_5,\mathrm{D}_{20}\}$&\\
   &$\{\mathrm{A}_5,\mathrm{A}_5,\mathrm{D}_{20},\mathrm{D}_{20}\}$&\\
  \hline
\end{tabular}
\end{center}
  \caption{Properties of length four irredundant generating sequences of $\mathrm{PSL}(2,p)$ for $p=19$ and $p=31$.}
    \label{tab:res22}
  \end{table}

\section{Acknowledgements}

This work would not have been possible without countless discussions with R. K. Dennis.  In particular, Lemmas~\ref{funnysubgroup},~\ref{unique}, and Cor.~\ref{cor:p+1} were directly discovered during such discussions and others were byproducts.  The author is grateful for R. K. Dennis' encouragement and academic support. 





\bibliographystyle{plain}	
\bibliographystyle{apsrev4-1}

%







\end{document}